\documentclass[psreqno,reqno,11pt]{amsart}
\usepackage{eurosym}
\usepackage{amssymb}
\usepackage{tikz}
\usepackage{amsmath}
\usepackage{tikz}
\usepackage{pgflibraryplotmarks}
\usepackage{wasysym}
\usepackage{stmaryrd}
\usepackage[english]{babel}

\usepackage{hyperref}

\theoremstyle{plain}
\newtheorem{theorem}{Theorem}[section]
\newtheorem{proposition}[theorem]{Proposition}
\newtheorem{lemma}[theorem]{Lemma}
\newtheorem{corollary}[theorem]{Corollary}

\theoremstyle{definition}

\newtheorem{remark}[theorem]{Remark}
\newtheorem{notation}[theorem]{Notation}
\newtheorem*{Acknowledgement}{Acknowledgment}
\def\N{\mathbb{N}}

\def\C{\mathbb{C}}
\def\P{\mathbb{P}}
\def\Q{\mathbb{Q}}

\def\e{\epsilon}

\begin{document}
\def\sect#1{\section*{\leftline{\large\bf #1}}}
\def\th#1{\noindent{\bf #1}\bgroup\it}
\def\endth{\egroup\par}

\title[Series and Power Series on Universally Complete Complex Vector Lattices]{
	Series and Power Series on Universally Complete Complex Vector Lattices}
\author{M. Roelands}
\address{Unit for BMI, North-West University, Private Bag X6001, Potchefstroom,
	2520, South Africa}
\email{mark.roelands@gmail.com}
\author{C. Schwanke}
\address{Unit for BMI, North-West University, Private Bag X6001, Potchefstroom,
	2520, South Africa}
\email{cmschwanke26@gmail.com}
\date{\today}
\subjclass[2010]{Primary: 46A40; Secondary: 40J05}
\keywords{series, power series, complex vector lattice, order convergence, complex analysis}

\begin{abstract}
In this paper we prove an $n$th root test for series as well as a Cauchy-Hadamard type formula and Abel's' theorem for power series on universally complete Archimedean complex vector lattices. These results are aimed at developing an alternative approach to the classical theory of complex series and power series using the notion of order convergence.
\end{abstract}

\maketitle

\section{Introduction}\label{S:intro}

The theory of complex analysis on complex Banach spaces enjoys a long, well-known, and fruitful history, which is summarized in the text \cite{Muj}. This theory generalizes the classical theory of functions of a complex variable by using norms as a generalization of the ordinary modulus on the complex plane. In this paper we investigate complex analysis via an alternative generalization of the complex modulus on $\mathbb{C}$: \textit{the modulus on Archimedean complex vector lattices}. This complex vector lattice modulus naturally leads to the notion of order convergence, providing us with an order-theoretic perspective of classical complex analysis.

The idea of an abstract order-theoretic complex modulus, to the authors' best knowledge, first appeared on complex AL-spaces in a 1963 paper by Rieffel (see \cite{Rieffel1,Rieffel2}). Here Rieffel defined the modulus for any element $z$ of a complex $AL$-space by
\begin{equation}\label{eq:mod}
|z|:=\sup\{Re(e^{i\theta}z):\theta\in[0,2\pi]\}.
\end{equation}
In 1968, this formula for the modulus was defined more generally on Banach lattices by Lotz \cite{Lotz} and was further extended to vector space complexifications of uniformly complete Archimedean vector lattices in 1971 by Luxemburg and Zaanen \cite{LuxZan2}. The authors of \cite{LuxZan2} referred to these complexifications of uniformly complete Archimedean vector lattices as \textit{complex Riesz spaces}.

The notion of a complex Riesz space, or complex vector lattice, was extended slightly by de Schipper in 1973 (see \cite{dS}), who defined a complex vector lattice to be a complex vector space $E$ of the form $E=E_\rho\oplus iE_\rho$, where $E_\rho$ is a real vector lattice that is closed under the supremum in \eqref{eq:mod}. This extension was not a trivial one, as there exist real vector lattices that are closed under the supremum in \eqref{eq:mod} that are not uniformly complete. An example provided by Azouzi in \cite{Az} is the space of all real-valued step functions. A trivial variation of the formula \eqref{eq:mod} was termed the \textit{de Schipper formula} in \cite{AzBoBus}.

Providing an alternative approach, Mittelmeyer and Wolff in \cite{MW} defined a modulus on a complex vector space $E$ to be an idempotent mapping $|\cdot|\colon E\to E$ that satisfies
\begin{itemize}
	\item[(i)] $|\alpha x|=|\alpha||x|$ for all $\alpha\in\C$ and each $x\in E$,
	\item[(ii)] $\Bigl|\bigl||x|+|y|\bigr|-|x+y|\Bigr|=|x|+|y|-|x+y|$, and
	\item[(iii)] $E$ is in the $\C$-linear hull of the range of $|\cdot|$. 
\end{itemize}
The authors of \cite{MW} proved that any complex vector space $E$ equipped with a modulus is a complex vector lattice (in the sense of de Schipper), and that the modulus on $E$ was necessarily given by the de Schipper formula.

We focus here on using this complex vector lattice modulus to abstract the classical $n$th root test for series as well as the Cauchy-Hadamard formula and Abel's Theorem for power series on universally complete Archimedean complex vector lattices (see Theorem~\ref{T:nthroottest}, Theorem~\ref{T:CH}, and Proposition~\ref{P:Abel's}). Indeed, universally complete Archimedean complex vector lattices provide a natural framework for order-theoretic generalizations of these classical results. In particular, these spaces possess a vector lattice complex modulus, giving us notions of magnitude and order convergence. Additionally, it is well known that every universally complete Archimedean (real) vector lattice is in fact a $\Phi$-algebra, i.e. an $f$-algebra with a multiplicative identity. It follows that universally complete Archimedean complex vector lattices possess an algebraic structure which is compatible with their order structure. Moreover, we can define $n$th roots of positive elements inside every universally complete Archimedean complex vector lattice, allowing us to formulate an abstract $n$th root test in this setting. We stress that the assumption of universal completeness is essential to this paper. For example, Remark~\ref{R:notuc} illustrates that Lemma~\ref{L:geoser} and Theorem~\ref{T:nthroottest}(i) fail to hold inside a non-universally complete Archimedean complex $\Phi$-algebra.

To our best knowledge, this is the first paper that studies power series abstractly on universally complete Archimedean complex vector lattices. The main results here illustrate how fundamental aspects of vector lattice theory, such as weak order units and projection bands, play a significant yet almost hidden role in classical complex analysis.

\section{Preliminaries}\label{S:prelims}

We refer the reader to the standard texts \cite{AB,LuxZan1,Zan2} for any unexplained terminology or basic results regarding vector lattices, $f$-algebras, and $\Phi$-algebras. As conventional, $\mathbb{C}$ represents the field of complex numbers in this paper. The set of strictly positive integers is represented by $\mathbb{N}$, and we use the notation $\N_0:=\N\cup\{0\}$ throughout.

We say that an Archimedean (real) vector lattice $E$ is \textit{square mean closed} (see \cite[page 482]{AzBoBus} or \cite[page 356]{dS}) if $\sup\{ (\cos\theta)f+(\sin\theta)g:\theta\in[0,2\pi]\}$
exists in $E$ for every $f,g\in E$, in which case we write
\[
f\boxplus g:=\sup\{ (\cos\theta)f+(\sin\theta)g:\theta\in[0,2\pi]\}=\sup\{Re\bigl(e^{-i\theta}(f+ig)\bigr):\theta\in[0,2\pi]\}.
\]

For example, every uniformly complete Archimedean vector lattice is square mean closed (see \cite[Section 2]{BeuHuidP}). Given a square mean closed Archimedean (real) vector lattice $E$, we call the vector space complexification $E\oplus iE$ an \textit{Archimedean complex vector lattice} \cite[pages 356--357]{dS}. The vector lattice $E$ is called the \textit{real part} of $E\oplus iE$. Given an Archimedean complex vector lattice $E$, we denote its real part by $E_\rho$.

The \textit{modulus} on an Archimedean complex vector lattice $E$ is defined by
\[
|f+ig|:=f\boxplus g\quad (f,g\in E_\rho).
\]

The notion of the positive cone of a vector lattice naturally extends to complex vector lattices. Indeed, the \textit{positive cone} of an Archimedean complex vector lattice $E$ is the set $E_{+}:=\{ f\in E:|f|=f\}$. Note that $E_+$ is also the positive cone of $E_\rho$.

A key observation to be made is that every universally complete Archimedean (real) vector lattice is a $\Phi$-algebra (see e.g. \cite[Theorem~15.22]{dJvR}), allowing us to introduce the notions of power series and $n$th roots in these spaces. Recall that a \textit{(real) $\Phi$-algebra} is a (real) $f$-algebra possessing a multiplicative identity. The multiplication on a universally complete Archimedean complex vector lattice $E$ will always be indicated by juxtaposition, and we denote the multiplicative identity of $E$ by $e$ throughout. Observe that the Dedekind completeness of $E$ implies that $E$ is square mean closed, and also note that the multiplication on $E$ canonically extends to $E\oplus iE$. Therefore, $E\oplus iE$ is an Archimedean complex vector lattice endowed with a complex $\Phi$-algebra structure. For more information on complex $\Phi$-algebras, see \cite{BusSch3}.

Given an Archimedean (real) $f$-algebra $A$, $a\in A_+$, and $n\in\N$, if there exists $r\in A_+$ such that $r^n=a$, we call $r$ an \textit{$n$th root} of $a$. In a uniformly complete Archimedean complex $\Phi$-algebra $A$, there exists a unique $n$th root of $a$ for every $a\in A_+$ and all $n\in\N$ (see \cite[Corollary~6]{BeuHui}). This $n$th root of $a$ is denoted by $a^{1/n}$. 

We proceed by recording some fundamental properties of an Archimedean complex $\Phi$-algebra $A$ that we will repeatedly use throughout the next section. Real $\Phi$-algebra analogues of statements (i)--(iv) can be found in \cite[Section 142]{Zan2}, and (v) is a consequence of (iv), which follows from \cite[Theorem 2.37]{AB}. Finally, \cite[Theorem 2.44]{AB} and \cite[Theorem~18.13]{LuxZan1} yield (vi).

For all $a,b\in A$ and any order projection $\P$ on $A$,

\begin{itemize}
	\item[(i)] $ab=ba$.
	\item[(ii)] $|ab|=|a||b|$.
	\item[(iii)] $|a|\wedge|b|=0$ if and only if $ab=0$.
	\item[(iv)] If $a$ is invertible and positive then $a^{-1}$ is positive.
	\item[(iv)] $\P(ab)=a\P(b)$.
	\item[(v)] $\P(ab)=\P(a)\P(b)$.
	\item[(vi)] $\P$ is an order continuous normal Riesz homomorphism.
\end{itemize}

We conclude this section by introducing the concept of order convergence on Archimedean complex vector lattices, which we use to determine the convergence of series on these spaces. Recall that a sequence $(x_n)_{n\geq 0}$ in a complex vector lattice \textit{converges in order} to $x$, which we denote by $x_n\to x$, if there exists a positive decreasing sequence $(p_n)_{n\geq 0}$ with $\inf_{n\geq0}p_n=0$ (in symbols $p_n\downarrow 0$) and $|x_n-x|\leq p_n$ for all $n\in\N_0$. We call a sequence $(x_n)_{n\geq 0}$ an \textit{order Cauchy sequence} if there exists a sequence $(p_n)_{n\geq 0}$ satisfying $p_n\downarrow 0$ for which $|x_n-x_m|\leq p_n$ for all $m,n\in\N_0$ with $n\leq m$. A complex vector lattice $E$ with the property that every order Cauchy sequence converges in order in $E$ is called \textit{order Cauchy complete}. Every Dedekind complete complex vector lattice is order Cauchy complete (see e.g. \cite[Exercise~10.6]{LilZan}).

Given an order bounded sequence $(x_n)_{n\geq 0}$ in a Dedekind $\sigma$-complete (real) vector lattice $E$, we define
\[
\limsup_{n\to\infty}x_n:=\inf_{n\geq 0}\sup_{m\geq n}x_n\qquad \text{and}\qquad \liminf_{n\to\infty}x_n:=\sup_{n\geq 0}\inf_{m\geq n}x_n.
\]
Note that $\limsup_{n\to\infty}x_n$ and $\liminf_{n\to\infty}x_n$ are in fact the order limits of the sequences $(\sup_{m\geq n}x_n)_{n\geq 0}$ and $(\inf_{m\geq n}x_n)_{n\geq 0}$, respectively.

\section{Series and Power Series}\label{S:results}

A \textit{series} in an Archimedean complex vector lattice $E$ is an infinite sum $\sum_{n=0}^{\infty}a_n$, where $a_n\in E$ for all $n\in\N_0$. We say $\sum_{n=0}^{\infty}a_n$ \textit{converges in order} if the sequence $\left(\sum_{n=0}^{m}a_n\right)_{m\geq 0}$ of partial sums has an order limit in $E$. Moreover, $\sum_{n=0}^{\infty}a_n$ is said to \textit{converge absolutely in order} if $\sum_{n=0}^{\infty}|a_n|$ converges in order. It is readily verified that if $E$ is order Cauchy complete then any series in $E$ that converges absolutely in order also converges in order.

We proceed with the divergence test for series in Archimedean complex vector lattices, whose standard proof is left to the reader.

\begin{lemma}[\textbf{The Divergence Test}]\label{L:divtest}
	Let $E$ be an Archimedean complex vector lattice. If the series $\sum_{n=0}^{\infty}a_n$ in $E$ converges in order then $a_n\to0$.
\end{lemma}

The following notation will be used frequently throughout the remainder of the paper.

\begin{notation}\label{N:ll}
	Given an Archimedean (real) vector lattice $E$ and $x,y\in E$, we write $x\ll y$ if $(y-x)^+$ is a weak order unit in $E$.
\end{notation}

For example, in the space $\ell_\infty(\N)$ of all bounded sequences, two elements $x,y\in\ell_\infty(\N)$ satisfy $x\ll y$ if and only if $x_n<y_n$ for every $n\in\N$.

\begin{remark}\label{R:wou^-1}
	Every positive invertible element in an Archimedean (real) $\Phi$-algebra is a weak order unit \cite[Theorem~142.2(ii)]{Zan2}. The converse however is not true in general. For example, the sequence $(x_n)_{n\geq 1}$ defined by $x_n=\frac{1}{n}\ (n\in\N)$ is a weak order unit in the space of all convergent real-valued sequences $c$ but is not invertible in $c$. However, if $E$ is a universally complete Archimedean (real) vector lattice then $E$ can be identified with $C^\infty(X)$ for some $X$ that is an extremely disconnected compact Hausdorff space, see \cite[Theorem~15.22]{dJvR}. Using this identification, we can easily see that every weak order unit $u$ in a universally complete Archimedean complex vector lattice $E$ is invertible. Indeed, if $u$ is a weak order unit of $C^\infty(X)$, we have from \cite[Theorem~1.38]{AB} that
	\[
	\textbf{1}=\sup_{n\geq 1}\{\textbf{1}\wedge nu\},
	\]
	where $\textbf{1}$ is the constant function on $X$ taking the value $1$. It follows that there does not exist a dense subset $Y$ of $X$ for which $u(x)=0$ for all $x\in Y$. Hence $u$ has a continuous inverse in $C^\infty(X)$. 
\end{remark}

We proceed by providing a useful characterization of the relation $\ll$ introduced in Notation~\ref{N:ll} in the context of universally complete Archimedean (real) vector lattices.

\begin{lemma}\label{L:ll}
Let $E$ be a universally complete Archimedean (real) vector lattice, and let $x,y\in E$. Then $x\ll y$ if and only if $x\leq y$ and $y-x$ is invertible.
\end{lemma}

\begin{proof}
Suppose $x\ll y$, so $(y-x)^+$ is a weak order unit. From Remark~\ref{R:wou^-1}, $(y-x)^+$ is invertible. Moreover, we have from \cite[Theorem~1.38]{AB} that
\begin{align*}
(y-x)^-=\sup_{n\in\N}\ (y-x)^-\wedge n(y-x)^+=0.
\end{align*}
Thus $x\leq y$.

Conversely, assume $x\leq y$ and that $y-x$ is invertible. Then $(y-x)^+=y-x$ is a weak order unit by Remark~\ref{R:wou^-1}.
\end{proof}

The following lemma regarding the convergence of the geometric series will play a crucial role in the proof of the $n$th root test Theorem~\ref{T:nthroottest} and Abel's theorem Proposition~\ref{P:Abel's}. In a universally complete Archimedean (real or complex) vector lattice $E$ and $a\in E$ we adopt the standard convention that $a^0=e$.

\begin{lemma}\label{L:geoser}
Let $E$ be a universally complete Archimedean complex vector lattice. The geometric series $\sum_{n=0}^{\infty}a^n$ in $E$ converges absolutely in order if and only if $|a|\ll e$. In this case, $\sum_{n=0}^{\infty}a^n=(e-a)^{-1}$.
\end{lemma}

\begin{proof}
Suppose $\sum_{n=0}^{\infty}a^n$ converges absolutely in order. Note that for $m\in\N_0$,
\begin{equation}\label{eq:geo}
(e-|a|)\sum_{n=0}^{m}|a|^n=e-|a|^{m+1}.
\end{equation}
Moreover,
\[
(e-|a|)\sum_{n=0}^{m}|a|^n\to(e-|a|)\sum_{n=0}^{\infty}|a|^n,
\]
while by Lemma~\ref{L:divtest},
\[
e-|a|^{m+1}\to e.
\]
Since order limits are unique, we have
\[
(e-|a|)\sum_{n=0}^{\infty}|a|^n=e.
\]
It now follows from Lemma~\ref{L:ll} that $|a|\ll e$.

On the other hand, if $|a|\ll e$ then $|a|\leq e$ and $e-|a|$ is invertible by Lemma~\ref{L:ll}. Thus for each $m\in\N_0$,
\[
\sum_{n=0}^{m}|a|^n=(e-|a|)^{-1}(e-|a|^{m+1})\leq(e-|a|)^{-1},
\]
hence $\sum_{n=0}^{\infty}|a|^n$ converges in order, as $E$ is Dedekind complete. It follows from Lemma~\ref{L:divtest} that $|a|^n\to 0$. Hence, in light of \eqref{eq:geo}, we obtain $\sum_{n=0}^{\infty}|a|^n=(e-|a|)^{-1}$.

Finally, since $E$ is order Cauchy complete, $\sum_{n=0}^{\infty}a^n$ also converges in order. A similar argument shows that $\sum_{n=0}^{\infty}a^n=(e-a)^{-1}$.
\end{proof}

Next we provide an order-theoretic version of the classical $n$th root test.

\begin{theorem}[\textbf{$\mathbf{n}$th Root Test}]\label{T:nthroottest}
Let $E$ be a universally complete Archimedean complex vector lattice, and let $\sum_{n=0}^{\infty}a_n$ be a series in $E$.
\begin{itemize}
\item[(i)] If the sequence $(|a_n|^{1/n})_{n\geq 1}$ is order bounded and $\limsup_{n\to\infty}|a_n|^{1/n}\ll e$ then the series $\sum_{n=0}^{\infty}a_n$ converges absolutely in order.
\item[(ii)] If the series $\sum_{n=0}^{\infty}a_n$ converges in order then the sequence $(|a_n|^{1/n})_{n\geq 1}$ is order bounded and $\limsup_{n\to\infty}|a_n|^{1/n}\leq e$.
\item[(iii)] If $(|a_n|^{1/n})_{n\geq 1}$ is order bounded and we have that both $\limsup_{n\to\infty}|a_n|^{1/n}\not\ll e$ and $\limsup_{n\to\infty}|a_n|^{1/n}\leq e$ hold then this test is inconclusive.
\end{itemize}
\end{theorem}

\begin{proof}
(i) Suppose the sequence $(|a_n|^{1/n})_{n\geq 1}$ is order bounded, and also assume that
\[
L:=\limsup_{n\to\infty}|a_n|^{1/n}\ll e.
\]
Put $b_m:=\sup_{n\geq m}|a_n|^{1/n}$ for $m\in\N$, and set $b_0:=e$. For each $m\in\N_0$ define $B_{b_m<e}$ to be the principal band generated by $(e-b_m)^+$, and set $\P_{b_m<e}$ to be the corresponding order projection onto $B_{b_m<e}$. Moreover, set $B_{b_m\geq e}$ to be the disjoint complement of $B_{b_m<e}$ and denote the corresponding order projection by $\P_{b_m\geq e}$ for each $m\in\N_0$. Define
\[
\Q_m:=\P_{b_m<e}-\P_{b_{m-1}<e}\quad (m\in\N).
\]
Then for $m,p\in\N$ with $m<p$ and $a\in E_+$ we have
\begin{align*}
0\leq\Q_m(a)\wedge\Q_p(a)&=\bigl(\P_{b_m<e}(a)-\P_{b_{m-1}<e}(a)\bigr)\wedge\bigl(\P_{b_p<e}(a)-\P_{b_{p-1}<e}(a)\bigr)\\
&\leq\P_{b_{p-1}<e}(a)\wedge\bigl(a-\P_{b_{p-1}<e}(a)\bigr)\\
&=0.
\end{align*}
Thus $(\Q_m)_{m\geq 1}$ is a pairwise disjoint sequence of order projections. Next we claim that $\P_{b_m<e}(b_m)\ll e$ for all $m\in\N_0$. Indeed, for $m\in\N_0$ we have
\begin{equation}\label{eq:ineqs}
e-\P_{b_m<e}(b_m)\geq\P_{b_m<e}(e-b_m)=(e-b_m)^+\geq 0,
\end{equation}
where the above equality follows from \cite[page 215]{LilZan}. Also, note that
\begin{align*}
e-\P_{b_m<e}(b_m)&=\P_{b_m<e}\bigl(e-\P_{b_m<e}(b_m)\bigr)+\P_{b_m\geq e}\bigl(e-\P_{b_m<e}(b_m)\bigr)\\
&=\P_{b_m<e}\bigl(e-\P_{b_m<e}(b_m)\bigr)+\P_{b_m\geq e}(e).
\end{align*}
Note that by $\eqref{eq:ineqs}$, we have that $\P_{b_m<e}(e-\P_{b_m<e}(b_m))\geq\P_{b_m<e}(e-b_m)^+=(e-b_m)^+$, which shows that
$\P_{b_m<e}(e-\P_{b_m<e}(b_m))$ is a weak order unit in $B_{b_m<e}$. Since we also have that $\P_{b_m\geq e}(e)$ is a weak order unit in $B_{b_m\geq e}$, we conclude that $e-\P_{b_m<e}(b_m)$ is a weak order unit. Thus $\P_{b_m<e}(b_m)\ll e$, as claimed.

Next let $m,p\in\N$ with $m\leq p$. Recalling that $b_m:=\sup_{n\geq m}|a_n|^{1/n}$ for $m\in\N$, we obtain
\begin{align*}
\sum_{n=0}^{p}\Q_m(|a_n|)&\leq\sum_{n=0}^{p}\P_{b_m<e}(|a_n|)=\sum_{n=0}^{m-1}\P_{b_m<e}(|a_n|)+\sum_{n=m}^{p}\P_{b_m<e}(|a_n|)\\
&\leq\sum_{n=0}^{m-1}\P_{b_m<e}(|a_n|)+\sum_{n=m}^{\infty}\bigl(\P_{b_m<e}(b_m)\bigr)^n,
\end{align*}
and we stress that $\sum_{n=m}^{\infty}\bigl(\P_{b_m<e}(b_m)\bigr)^n$ converges in order by Lemma~\ref{L:geoser}. Hence
\[
\sup_{p\geq m}\sum_{n=0}^{p}\Q_m(|a_n|)\in E,
\]
and it follows that $\sum_{n=0}^{\infty}\Q_m(|a_n|)$ converges in order for all $m\in\N$. Also, for all $m,p\in\N$ with $m\neq p$ and each $k,l\in\N_0$ we have $\Q_m(|a_k|)\perp\Q_p(|a_l|)$. Thus $\Q_m(|a_k|)\Q_p(|a_l|)=0$ for all such $k,l,m,p$. It follows that for $t_{1},t_2\in\N$,
\[
\sum_{n=0}^{t_1}\Q_m(|a_n|)\sum_{n=0}^{t_2}\Q_p(|a_n|)=0,
\]
and by the order continuity of $f$-algebra multiplication, with the limits below taken to be order limits, we have for $m\neq p$ that
\[
\sum_{n=0}^{\infty}\Q_m(|a_n|)\sum_{n=0}^{\infty}\Q_p(|a_n|)=\lim_{t_1,t_2\to\infty}\sum_{n=0}^{t_1}\Q_m(|a_n|)\sum_{n=0}^{t_2}\Q_p(|a_n|)=0.
\]
Thus $\sum_{n=0}^{\infty}\Q_m(|a_n|)\perp\sum_{n=0}^{\infty}\Q_p(|a_n|)$ holds for all $m\neq p$, and hence
\[
\sup_{m\geq 1}\sum_{n=0}^{\infty}\Q_m(|a_n|)\in E.
\]
Furthermore, for any $k,p\in\N$ we have $\Q_k\left(\sum_{n=0}^{p}|a_n|\right)\leq\sup_{m\geq 1}\sum_{n=0}^{\infty}\Q_m(|a_n|)$. Hence
\begin{equation}\label{eq:Qineq}
\sup_{k\geq 1}\Q_k\left(\sum_{n=0}^{p}|a_n|\right)\leq\sup_{m\geq 1}\sum_{n=0}^{\infty}\Q_m(|a_n|).
\end{equation}
Using \cite[Theorem~8.2(1)]{LilZan} in the second equality below, we obtain, with the limits below taken to be order limits,
\begin{align*}
\sup_{k\geq 1}\Q_k\left(\sum_{n=0}^{p}|a_n|\right)&=\lim_{m\to\infty}\sup_{k\leq m}\Q_k\left(\sum_{n=0}^{p}|a_n|\right)=\lim_{m\to\infty}\sum_{k=1}^{ m}\Q_k\left(\sum_{n=0}^{p}|a_n|\right)\\
&=\lim_{m\to\infty}\P_{b_m<e}\left(\sum_{n=0}^{p}|a_n|\right).
\end{align*}
Next note that $\bigl((e-b_m)^+\bigr)_{m\geq 1}$ is an increasing sequence with supremum $(e-L)^+=e-L$, which is by assumption a weak order unit. Also, for each $m\in\N$ the band $B$ generated by $\bigcup_{m=1}^{\infty}B_{b_m<e}$ contains $(e-b_m)^+$, so $e-L\in B$, as bands are order closed. Therefore, $B=E$. It now follows from \cite[Theorems~30.4 and 30.5(ii)]{LuxZan1} that
\[
\lim_{m\to\infty}\P_{b_m<e}\left(\sum_{n=0}^{p}|a_n|\right)=\sum_{n=0}^{p}|a_n|.
\]
Therefore, by \eqref{eq:Qineq} and the string of equalities which follow, we have 
\[
\sum_{n=0}^{p}|a_n|\leq\sup_{m\geq 1}\sum_{n=0}^{\infty}\Q_m(|a_n|)
\]
for every $p\in\N$. By the Dedekind completeness of $E$, we have that
\[
\sup_{p\geq 1}\sum_{n=0}^{p}|a_n|\in E.
\]
We conclude that $\sum_{n=0}^{\infty}|a_n|$ converges in order, and in particular,
\[
\sum_{n=0}^{\infty}|a_n|=\sup_{m\geq 1}\sum_{n=0}^{\infty}\Q_m(|a_n|).
\]

(ii) Suppose $\sum_{n=0}^{\infty}a_n$ converges in order. Then $a_n\to 0$ by Lemma~\ref{L:divtest}. Thus there exists a sequence $p_n\downarrow 0$ such that $|a_n|\leq p_n$ for all $n\in\N_0$. By the proof of \cite[Lemma~4.2]{BusSch3}, we have for any $n\in\N$,
\begin{align*}
|a_n|^{1/n}&=\inf\left\{\textstyle{\frac{1}{n}}\theta_{1}|a_n|+(1-\frac{1}{n})\theta_{2} e:\theta_{1},\theta_{2}\in(0,\infty),\ \theta_{1}^{1/n}\theta_{2}^{1-1/n}=1\right\}\\
&\leq\textstyle{\frac{1}{n}|a_n|+(1-\frac{1}{n})e}.
\end{align*}
Hence for all $m,n\in\N$ with $m\geq n$ we have $\sup_{n\geq m}|a_n|^{1/n}\leq\textstyle{\frac{1}{m}}p_m+e$. Taking $m\to\infty$ now yields $\limsup_{n\to\infty}|a_n|^{1/n}\leq e$.

(iii) For any Archimedean complex $\Phi$-algebra $A$, the series $\sum_{n=1}^{\infty}\frac{1}{n^2}e$ converges in order, while
\[
\limsup_{n\to\infty}\left(\frac{1}{n^2}e\right)^{1/n}=e.
\]  

On the other hand, consider the universally complete Archimedean complex vector lattice $E:=\C^{\C}$. For each $n\in\N$ define $f_n:=\chi_{\bar{D}_{1/n}(0)}$, where $\chi_{\bar{D}_{1/n}(0)}$ is the characteristic function on the closed disk $\bar{D}_{1/n}(0)$ centered at $0$ with radius $1/n$. Then $|f_n|^{1/n}=|f_n|$ for every $n\in\N$ and 
\[
L:=\lim\sup_{n\to\infty}|f_n|^{1/n}=\chi_{\{0\}}\not\ll\bf{1}
\]
and $L\leq\bf{1}$, where again $\bf{1}$ denotes the constant function on $\C$ taking the value $1$. Note that for each $m\in\N$ we have $\sum_{n=1}^{m}f_n(0)=m$, and therefore the series $\sum_{n=1}^{\infty}f_n$ diverges.
\end{proof}

The following remark illustrates the necessity of the universal completeness assumption in both Lemma~\ref{L:geoser} and Theorem~\ref{T:nthroottest}(i).

\begin{remark}\label{R:notuc}
Consider the non-universally complete Archimedean complex $\Phi$-algebra $A$ of complex-valued bounded continuous functions defined on $(0,1)$. Let $f\in A$ be the identity function on $(0,1)$. Observe that $f=|f|\ll\textbf{1}$ but the series $\sum\limits_{n=0}^{\infty}f^n$ does not converge in order in $A$. 
\end{remark}

A \textit{power series} on a universally complete Archimedean complex vector lattice $E$, centered at $c\in E$, is a series of the form $S(z):=\sum_{n=0}^{\infty}a_n(z-c)^n$, where $a_n\in E$ for each $n\in\N_0$ and $z$ represents a variable in $E$. We say a power series $S(z)$ \textit{converges uniformly in order} on a region $D\subseteq E$ if there exists a sequence $p_m\downarrow 0$ such that
\[
\sup_{z\in D}\left|\sum_{n=0}^{m}a_n(z-c)^n-S(z)\right|\leq p_m
\]
holds for each $m\in\N_0$.

Given $c\in E$ and $r\in E_+$, we use the notation $\bar{\Delta}(c,r):=\{z\in E:|z-c|\leq r\}$ to indicate the order closed ball centered at $c$ with radius $r$. For a power series $S(z)=\sum_{n=0}^{\infty}a_n(z-c)^n$ we define 
\[
\Omega_S:=\left\{r\in E_+\colon S(z)\ \textnormal{converges uniformly in order on}\ \bar{\Delta}(c,r)\right\}.
\]
If $\Omega_S$ is order bounded then $\sup\Omega_S$ exists in $E_+$, and we call $\rho_S:=\sup \Omega_S$ the \textit{radius of convergence} of $S(z)$.

For example, if we consider the geometric series $G(z)=\sum_{n=0}^{\infty}z^n$ on a universally complete Archimedean complex vector lattice then
for $r\ll e$ and $z\in\bar{\Delta}(0,r)$ we have that
\[
\left|\sum_{n=0}^mz^n-G(z)\right|\leq\sum_{n=m+1}^{\infty}|z|^n\leq\sum_{n=m+1}^\infty r^n\downarrow_m 0.
\]
Hence $G(z)$ converges uniformly in order on $\bar{\Delta}(0,r)$, and thus $r\in\Omega_G$. On the other hand, by Lemmas~\ref{L:ll}~and~\ref{L:geoser}, we know $r\leq e$ for all $r\in\Omega_G$. Thus $\Omega_G$ is order bounded and $\rho_G\leq e$. Furthermore, it follows from the above argument that $\e e\in\Omega_G$ for all $0\leq\e<1$, showing that $\rho_G=e$.

The following lemma will be useful throughout the remainder of the paper.

\begin{lemma}\label{L:supinflemma}
	Let $A$ be a uniformly complete Archimedean complex $\Phi$-algebra. For all $B\subseteq A_+$ such that $\sup B$ exists in $A_+$, respectively, $\inf B$ exists in $A_+$, we have
	\[
	\sup(aB)=a\sup B,\quad \mbox{respectively}\quad \inf(aB)=a\inf B\quad (a\in A_+). 
	\]
\end{lemma}

\begin{proof}
Let $B\subseteq A_+$ be such that $\sup B$ exists in $A_+$, and let $a\in A_+$. For every $b\in B$ we have $ab\leq a\sup B$. Next let $ab\leq u$ for every $b\in B$. Then $(a+e)b\leq u+b\leq u+\sup B$. Note that $A$ is $e$-uniformly complete by \cite[Theorem~11.4]{dP}, and therefore it follows from \cite[Theorem~11.1]{dP} that $a+e$ is invertible in $A$. Hence $b\leq(a+e)^{-1}(u+\sup B)$. Therefore, $\sup B\leq(a+e)^{-1}(u+\sup B)$, which implies $a\sup B\leq u$. A similar argument proves the second identity.
\end{proof}

We next record some useful facts regarding the set $\Omega_S$ with respect to a power series $S$. 

\begin{proposition}\label{P:Omega}
Let $S(z):=\sum_{n=0}^\infty a_n(z-c)^n$ be a power series on a universally complete Archimedean complex vector lattice $E$, and let $\Omega_S$ be defined as above.
\begin{itemize}
\item[(i)] $\Omega_S$ is a solid sublattice of $E_+$.
\item[(ii)] If the sequence $(|a_n|^{1/n})_{n\geq 1}$ is order bounded then, denoting the band generated by $L:=\limsup_{n\to\infty}|a_n|^{1/n}$ by $B_L$ and its disjoint complement by $B^d_L$, we have the inclusion $(B^d_L)_+\subseteq\Omega_S$.
\end{itemize}
\end{proposition}

\begin{proof}
(i) It is clear from its definition that $\Omega_S$ is a solid set, and thus $\Omega_S$ is closed under finite infima. To prove that $\Omega_S$ is closed under finite suprema, let $r,s\in\Omega_S$. By using a change of variables if necessary, we can assume that $c=0$. Then $S(z)$ converges uniformly in order on $\bar{\Delta}(0,r)$ and also on $\bar{\Delta}(0,s)$. Using \cite[page 215]{LilZan} in the second identity below, we obtain
\[
\P_{r<s}(s)-\P_{r<s}(r)=\P_{r<s}(s-r)=(s-r)^+\geq 0.
\]
Thus
\[
\P_{r<s}(r\vee s)=\P_{r<s}(r)\vee\P_{r<s}(s)=\P_{r<s}(s).
\]
It follows that for $|z|\leq r\vee s$,
\[
s-\P_{r<s}(|z|)\geq s-\P_{r<s}(r\vee s)=s-\P_{r<s}(s)\geq 0,
\]
and thus $\P_{r<s}(|z|)\leq s$. Hence $\sum_{n=0}^{\infty}a_n\P_{r<s}(z)^n$ converges uniformly in order on $\bar{\Delta}(0,r\vee s)$. Similarly, $\sum_{n=0}^{\infty}a_n\P_{r\geq s}(z)^n$ also converges uniformly in order on $\bar{\Delta}(0,r\vee s)$. Finally, observe that for all $|z|\leq r\vee s$ and any $m\in\N_0$,
\begin{align*}
\sum_{n=0}^{m}a_nz^n&=\sum_{n=0}^{m}a_n\P_{r<s}(z)^n+\sum_{n=0}^{m}a_n\P_{r\geq s}(z)^n.
\end{align*}
It now follows that $S(z)$ converges uniformly in order on $\bar{\Delta}(0,r\vee s)$. Therefore, $r\vee s\in\Omega_S$.
	
(ii) Suppose $(|a_n|^{1/n})_{n\geq 1}$ is order bounded, and also set $L:=\limsup_{n\to\infty}|a_n|^{1/n}$. Let $r\in (B^d_L)_+$, and take $z\in E$ such that $|z-c|\leq r$. Since $L\perp|z-c|$, we have $L(z-c)=0$. Using Lemma~\ref{L:supinflemma} in the second equality below we obtain
\[
\limsup_{n\to\infty}(|a_n||z-c|^n)^{1/n}=\limsup_{n\to\infty}(|a_n|^{1/n}|z-c|)=L|z-c|=0\ll e.
\]
Hence $S(z)$ converges absolutely in order on $\bar{\Delta}(c,r)$ by Theorem~\ref{T:nthroottest}(i). Moreover, for each $m\in\N_0$ and all $z\in\bar{\Delta}(0,r)$,
\[
\left|\sum_{n=0}^ma_n(z-c)^n-S(z)\right|\leq\sum_{n=m+1}^\infty|a_n|r^n\downarrow_m0,
\]
and thus $S(z)$ converges uniformly in order on $\bar{\Delta}(c,r)$ as well. Therefore, $r\in\Omega_S$.
\end{proof}

Given a universally complete Archimedean complex vector lattice $E$, we denote the principle band generated by the element $a$ inside $E$ by $B_a$ and the associated band projection onto $B_a$ by $\P_a$. We note that, for each $a\in E\setminus\{0\}$, it is readily checked that $B_a$ is a universally complete $f$-subalgebra of $E$ that has $\P_a(e)$ as its multiplicative identity. Moreover, if we have $a\in E\setminus\{0\}$ then $a$ is a weak order unit in $B_a$ and thus has a multiplicative inverse in $B_a$.

\begin{notation}
For a universally complete Archimedean complex vector lattice $E$ and $a\in E$, $a\neq 0$, we denote the multiplicative inverse of $a$ in $B_a$ by $a^\ast$. Moreover, we set $0^\ast:=0$.
\end{notation}

We proceed by proving a universally complete Archimedean complex vector lattice version of the classical Cauchy-Hadamard formula.

\begin{theorem}[\textbf{Cauchy-Hadamard}]\label{T:CH}
Let $E$ be a universally complete Archimedean complex vector lattice, and let $S(z):=\sum_{n=0}^\infty a_n(z-c)^n$ be a power series on $E$.
\begin{itemize}
\item[(i)] If $(|a_n|^{1/n})_{n\geq 1}$ is order bounded then for $L:=\limsup_{n\to\infty}|a_n|^{1/n}$ we have that $\P_L(\Omega_S)$ is order bounded and $L^\ast=\underset{r\in\Omega_S}{\sup}\,\P_L(r)$.
\item[(ii)] If $\Omega_S$ is order bounded then $\bigl(\P_{\rho_S}(|a_n|^{1/n})\bigr)_{n\geq 1}$ is order bounded and furthermore we have $\limsup_{n\to\infty}\P_{\rho_S}(|a_n|^{1/n})=\rho_S^\ast$.
\item[(iii)] $(|a_n|^{1/n})_{n\geq 1}$ is order bounded and $L:=\limsup_{n\to\infty}|a_n|^{1/n}$ is a weak order unit if and only if $\Omega_S$ is order bounded and $\rho_S$ is a weak order unit. In either case, we have $L^{-1}=\rho_S$.
\item[(iv)] $\limsup_{n\to\infty}|a_n|^{1/n}=0$ if and only if $\Omega_S=E_+$.
\end{itemize}
\end{theorem}

\begin{proof}
(i) Suppose $(|a_n|^{1/n})_{n\geq 1}$ is order bounded and set $L:=\limsup_{n\to\infty}|a_n|^{1/n}$. If $L=0$ then statement (i) trivially holds, so we can assume $L\neq 0$. Let $0\leq\e<1$, and suppose $z\in E$ satisfies $|z-c|\leq\e L^\ast$. Using Lemma~\ref{L:supinflemma} in the second equality, we obtain
\begin{align*}
\limsup_{n\to\infty}|a_n(z-c)^n|^{1/n}&=\limsup_{n\to\infty}\left(|a_n|^{1/n}|z-c|\right)\leq\e\limsup_{n\to\infty}\left(|a_n|^{1/n}L^{\ast}\right)\\
&=\e LL^{\ast}=\e\P_L(e)\ll e.
\end{align*}
It follows from Theorem~\ref{T:nthroottest}(i) that $S(z)$ converges absolutely in order on $\bar{\Delta}(c,\e L^\ast)$. Since $E$ is order Cauchy complete, we know that $S(z)$ converges in order on $\bar{\Delta}(c,\e L^\ast)$. Moreover, for any $z\in\bar{\Delta}(c,\e L^\ast)$, we have
\[
\left|\sum_{n=0}^{m}a_n(z-c)^n-S(z)\right|\leq\sum_{n=m+1}^{\infty}|a_n|(\e L^\ast)^n\downarrow_m0,
\]
showing that $S(z)$ converges uniformly in order on $\bar{\Delta}(c,\e L^\ast)$. Hence $\e L^\ast\in\Omega_S$. Next let $r\in\Omega_S$ be arbitrary. Then by Lemma~\ref{L:supinflemma} we have 
\[
\limsup_{n\to\infty}\left(|a_n|^{1/n}r\right)=Lr,
\]
and we also know that $Lr\leq e$ from Theorem~\ref{T:nthroottest}(ii). Thus $L\P_L(r)\leq\P_L(e)$, and hence $\P_L(r)\leq L^{\ast}$ for all $r\in\Omega_S$. Hence $\P_L(\Omega_S)$ is order bounded. Set
\[
\rho_S':=\underset{r\in\Omega_S}{\sup}\,\P_L(r).
\]
Then $\rho_S'\leq L^{\ast}$. Moreover, since $\e L^\ast\leq\rho'_S$ holds for all $0\leq\e<1$, we also have $L^{\ast}\leq\rho_S'$.
	
(ii) Next suppose $\Omega_S$ is order bounded. If $\rho_S=0$ then statement (ii) of the theorem is trivial, so we can assume that $\rho_S>0$. Let $r\in\Omega_S$. By Theorem~\ref{T:nthroottest}(ii), we know
	\[
	\limsup_{n\to\infty}(|a_n|^{1/n}r)\leq e.
	\]
	Thus there exists a sequence $p_m\downarrow 0$ such that
	\begin{align*}
	\sup_{n\geq m}(|a_n|^{1/n}r)&=\sup_{n\geq m}(|a_n|^{1/n}r)-\limsup_{n\to\infty}(|a_n|^{1/n}r)+\limsup_{n\to\infty}(|a_n|^{1/n}r)\\
	&\leq p_m+e
	\end{align*}
	holds for each $m\in\N$. Thus for every $m\in\N$ and all $n\geq m$, it follows from Lemma~\ref{L:supinflemma} that
	\[
	|a_n|^{1/n}\rho_S=\sup_{r\in\Omega_S}\{|a_n|^{1/n}r\}\leq p_m+e,
	\]
	and hence
	\[
	\P_{\rho_S}(|a_n|^{1/n})\rho_S\leq\P_{\rho_S}(p_m)+\P_{\rho_S}(e).
	\]
	Thus we obtain
	\[
	\sup_{n\geq m}\P_{\rho_S}(|a_n|^{1/n})\leq\P_{\rho_S}(p_m)\rho_S^{\ast}+\rho_S^{\ast},
	\]
	showing that $\bigl(\P_{\rho_S}(|a_n|^{1/n})\bigr)_{n\geq 1}$ is order bounded. Taking $m\to\infty$, we see that 
	\[
	L':=\limsup_{n\to\infty}\P_{\rho_S}(|a_n|^{1/n})\leq\rho_S^{\ast}.
	\]
	Next applying Proposition~\ref{P:Omega}~(ii) to the power series 
	\[
	R(w)=\sum_{n=0}^{\infty}\P_{\rho_S}(a_n)\bigl(w-\P_{\rho_S}(c)\bigr)^n
	\]
	on $B_{\rho_S}$, we have $(B^d_{L'})_+\subseteq\Omega_{R}$, where $B_{L'}$ is the principal band generated by $L'$ inside $B_{\rho_S}$. We next prove that $\Omega_{R}\subseteq\Omega_S$. To this end, let $r\in\Omega_{R}$. Then $R(w)$ converges uniformly in order on $\bar{\Delta}(\P_{\rho_S}(c),r)$. To show that $r\in\Omega_S$, suppose $z\in E$ satisfies $|z-c|\leq r$. Then $z-c\in B_{\rho_S}$, and thus for every $n\in\N$ we have that $a_n(z-c)^n\in B_{\rho_S}$ as well. Hence for each $m\in\N$ we have
	\[
	\sum_{n=1}^m a_n(z-c)^n=\sum_{n=1}^m\P_{\rho_S}(a_n)\bigl(\P_{\rho_S}(z)-\P_{\rho_S}(c)\bigr)^n.
	\]
	Next let $p_m\downarrow 0$ satisfy 
	\[
	\sup_{w\in\bar{\Delta}(\P_{\rho_S}(c),r)}\left|\sum_{n=0}^m\P_{\rho_S}(a_n)\bigl(w-\P_{\rho_S}(c)\bigr)^n-R(w)\right|\leq p_m\quad (m\in\N_0).
	\]
Then
\begin{align*}
&\left|\sum_{n=1}^ma_n(z-c)^n-\Bigl(R\bigl(\P_{\rho_S}(z)\bigr)-\P_{\rho_S}(a_0)\Bigr)\right|=\left|\sum_{n=0}^m\P_{\rho_S}(a_n)\bigl(\P_{\rho_S}(z)-\P_{\rho_S}(c)\bigr)^n-R\bigl(\P_{\rho_S}(z)\bigr)\right|\\
&\hskip 2.5cm \leq\sup_{w\in\bar{\Delta}(\P_{\rho_S}(c),r)}\left|\sum_{n=0}^m\P_{\rho_S}(a_n)\bigl(w-\P_{\rho_S}(c)\bigr)^n-R\bigl(w\bigr)\right|\leq p_m,
\end{align*}
and hence $S(z)$ converges uniformly in order on $\bar{\Delta}(c,r)$, where we have 
\[
S(z)=R\bigl(\P_{\rho_S}(z)\bigr)-\P_{\rho_S}(a_0)+a_0.
\]
This shows that $r\in\Omega_S$, and thus we obtain $(B^d_{L'})_+\subseteq\Omega_{S}$. Noting that $\Omega_S$ is order bounded, we see that $L'$ is a weak order unit in $B_{\rho_S}$. Applying (i) to the space $B_{\rho_S}$, we have that $L'=\rho_S^\ast$.
	
(iii) Suppose $(|a_n|^{1/n})_{n\ge 1}$ is order bounded and that  $L:=\limsup_{n\to\infty}|a_n|^{1/n}$ is a weak order unit. It follows from (i) that $\Omega_S$ is order bounded, that $\rho_S$ is a weak order unit and that $L^{-1}=\rho_S$. Conversely, if $\Omega_S$ is order bounded and $\rho_S$ is a weak order unit then it follows from (ii) that the sequence $(|a_n|^{1/n})_{n\ge 1}$ is order bounded and that $L:=\limsup_{n\to\infty}|a_n|^{1/n}$ is a weak order unit that satisfies $L^{-1}=\rho_S$.
	
(iv) Suppose $L:=\limsup_{n\to\infty}|a_n|^{1/n}=0$. Then $E_+=(B_L^d)_+\subseteq\Omega_S$ by Proposition~\ref{P:Omega}~(ii), and thus $\Omega_S=E_+$. Conversely, if $\Omega_S=E_+$ then Theorem~\ref{T:nthroottest}(ii) and the fact that $ne\in\Omega_S$ for all $n\in\N$ imply that $S(ne+c)$ converges in order, that $L$ exists in $E$, and that $nL\leq e$ for all $n\in\N$. By the Archimedean property of $E$, we conclude that $L=0$.
\end{proof}

In the classical theory of power series, Abel's theorem relates the limit of a power series to its series of coefficients. Given that the tools used in the classical theorem of Abel are elementary in nature, it is not surprising that the proof of this theorem (see e.g. \cite[Section~2.5]{Ahlf}) adapts to a more general version in universally complete Archimedean complex vector lattices. 

\begin{proposition}[\textbf{Abel's Theorem}]\label{P:Abel's}
Let $E$ be a universally complete Archimedean complex vector lattice, and assume that $S(z):=\sum_{n=0}^{\infty}a_nz^n$ is a power series with radius of convergence $e$. Suppose $\sum_{n=0}^{\infty}a_n$ converges in order. Then for any sequence $(z_k)_{k\geq 0}$ in $E$ such that 

\begin{itemize}
	\item[(i)] $z_k\to e$,
	\item[(ii)] for each $k\in\N_0$ we have $|z_k|\ll e$, and
	\item[(iii)] the set $\left\{|e-z_k|(e-|z_k|)^{-1}:k\in\N_0\right\}$ is order bounded,
\end{itemize}
we have $S(z_k)\to S(e)$.
\end{proposition}

\begin{proof}
First suppose that $\sum_{n=0}^{\infty}a_n=0$. Set $s_m=\sum_{n=0}^ma_n$ for each $m\in\N_0$. Since $s_m\to 0$, there exists a sequence $p_m\downarrow 0$ such that $|s_m|\leq p_m$ for all $m\in\N_0$. Using summation by parts, we obtain
\[
\sum_{n=0}^ma_nz^n=(e-z)\left(\sum_{n=0}^{m-1}s_nz^n\right)+s_mz^m.
\]
Hence, for any $r\in\Omega_S$ and all $z\in\bar{\Delta}(0,r)$ we have 
\[
\left|\sum_{n=0}^ma_nz^n-(e-z)\left(\sum_{n=0}^{m-1}s_nz^n\right)\right|\leq|s_m||z|^m\leq|s_m|\leq p_m,
\]
and so
\[
S(z)=(e-z)\sum_{n=0}^{\infty}s_nz^n.
\]
Next let $(z_k)_{k\geq 0}$ be a sequence satisfying $(i),\ (ii),\ \text{and}\ (iii)$ above, and say $|e-z_k|\leq q_k$ for each $k\in\N_0$, with $q_k\downarrow 0$. Let $\upsilon\in E_+$ be an upper bound of the set 
\[
\left\{|e-z_k|(e-|z_k|)^{-1}:k\in\N_0\right\}, 
\]
and let $m\in\N_0$ be arbitrary. Then for each $k\in\N_0$, we obtain from Lemma~\ref{L:geoser} that
\begin{align*}
|S(z_k)|&\leq|e-z_k|\sum_{n=0}^{m}|s_n||z_k|^n+|e-z_k|\sum_{n=m+1}^{\infty}|s_n||z_k|^n\\
&\leq q_k\sum_{n=0}^{m}|s_n|+p_{m+1}|e-z_k|(e-|z_k|)^{-1}\\
&\leq q_k\sum_{n=0}^{m}|s_n|+p_{m+1}\upsilon.
\end{align*}
Thus $(|S(z_k)|)_{k\ge 1}$ is order bounded. Furthermore, using Lemma~\ref{L:supinflemma} in the last equality below, we obtain
\[
0\leq\limsup_{k\to\infty}|S(z_k)|=\inf_{p\geq 1}\left(\sup_{k\geq p}|S(z_k)|\right)\leq\inf_{p\geq 1}\left\{q_p\left(\sum_{n=0}^m|s_n|\right)+p_{m+1}\upsilon\right\}=p_{m+1}\upsilon.
\]
Letting $m\to\infty$ shows that $\limsup_{k\to\infty}|S(z_k)|=0$. Hence $\liminf_{k\to\infty}|S(z_k)|=0$, and by \cite[Theorem~12.7]{LilZan}, we have $S(z_k)\to 0=S(e)$.

More generally, assume $\sum_{n=0}^{\infty}a_n=a$ for some $a\in E$. Define a power series on $E$ by $S'(z):=\sum_{n=0}^\infty a'_nz^n$, where $a'_0:=a_0-a$ and $a'_n:=a_n$ for all $n\in\N$. Then $\sum_{n=0}^\infty a'_n=0$, and by the argument above, $S'(z_k)\to 0$. This however is equivalent to $S(z_k)\to a=S(e)$.
\end{proof}

Note that $e$ in Proposition~\ref{P:Abel's} can be replaced by any positive weak order unit in $E$.

\begin{corollary}
Let $S(z):=\sum_{n=0}^\infty a_nz^n$ be a power series on a universally complete Archimedean complex vector lattice $E$ with radius of convergence $\rho_S$ which is a weak order unit. Suppose $S(\rho_S)$ converges in order, and assume that $(z_k)_{k\ge 0}$ is a sequence in $E$ such that
\begin{itemize}
	\item[(i)] $z_k\to \rho_S$,
	\item[(ii)] for each $k\in\N_0$ we have $|z_k|\ll \rho_S$, and
	\item[(iii)] the set $\left\{|\rho_S-z_k|(\rho_S-|z_k|)^{-1}:k\in\N_0\right\}$ is order bounded.
\end{itemize}
Then $S(z_k)\to S(\rho_S)$. 
\end{corollary}

\begin{proof}
Consider the power series $S'(z)=\sum_{n=0}^{\infty}a_n\rho_S^nz^n$. By using Lemma~\ref{L:supinflemma} as well as Theorem~\ref{T:CH}(iii), we obtain
\[
\limsup_{n\to\infty}|a_n|^{1/n}\rho_S=\left(\limsup_{n\to\infty}|a_n|^{1/n}\right)\rho_S=e.
\]
Hence, by Theorem~\ref{T:CH}(iii) again, we conclude that $\rho_{S'}=e$. Furthermore, note that $S'(e)$ converges in order. Next note that the sequence $(\rho_S^{-1}z_k)_{k\geq 0}$ satisfies the hypotheses $(i), (ii)$, and $(iii)$ in the statement of Proposition~\ref{P:Abel's}. By Proposition~\ref{P:Abel's}, we obtain
\[
S(z_k)=S'(\rho_S^{-1}z_k)\to S'(e)=S(\rho_S).
\]
\end{proof}

\begin{Acknowledgement}
This research was partially supported by the Claude Leon Foundation and by the DST-NRF Centre of Excellence in Mathematical and Statistical Sciences (CoE-MaSS) (second author). Opinions expressed and conclusions arrived at are those of the authors and are not necessarily to be attributed to the CoE-MaSS.
\end{Acknowledgement}

\bibliography{nthchf}
\bibliographystyle{amsplain}

\end{document}